\documentclass[11pt,twoside]{article}
\usepackage{amsfonts}
\usepackage{amsmath}
{\markboth{\sl  C. Abdelkefi and M.
Rachdi}{\sl Some results on the Hardy space $H^1_k$}
\newtheorem{theorem}{Theorem}[section]

\newtheorem{remark}{Remark}[section]
\newtheorem{example}{Example}[section]
\newtheorem{lemma}{Lemma}[section]

\newtheorem{cor}{Corollary}[section]
\numberwithin {equation}{section}
\newenvironment{proof}{\textbf{Proof.}} {\hfill $\Box$}
\begin{document}
\title{\LARGE\bf Some results on the Hardy space $H^1_k$ associated with the Dunkl operators}
\author{Chokri Abdelkefi and Mongi Rachdi
 \footnote{\small This work was
completed with the support of the DGRST research project LR11ES11
and the program CMCU 10G / 1503.}\\ \small  Department of
Mathematics, Preparatory
Institute of Engineer Studies of Tunis  \\ \small 1089 Monfleury Tunis, University of Tunis, Tunisia\\
  \small E-mail : chokri.abdelkefi@ipeit.rnu.tn \\  \small E-mail : rachdi.mongi@ymail.com}%
\date{}
\maketitle
\begin{abstract}
In the present paper, we investigate in Dunkl analysis, the action
of some fundamental operators on the atomic Hardy space $H^1_k$.
\end{abstract}
{\small\bf Keywords: }{\small Dunkl operators, Dunkl transform, atomic Hardy spaces, Riesz transform, Hardy operator.
\section{Introduction }
\par $ $
In the classical complex analysis, the Hardy spaces were introduced
by Hardy [10] to characterize boundary values of analytic functions
on the unit disk. The one-dimensional atomic decomposition of the
Hardy spaces is due to Coifman [4] and its higher-dimensional
extension to Latter [11]. Using the atomic decomposition Coifman and
Weiss [5] extended the definition of the Hardy spaces to more
general structures.

In this paper, we consider the harmonic analysis on the Euclidean
space $\mathbb{R}^{d}$ associated to the class of rational Dunkl
operators $T_i, 1 \leq i \leq d$. Introduced by C. F. Dunkl in [6],
these operators are commuting differential-difference operators
related to an arbitrary finite reflection group $G$ and a non
negative multiplicity function $k$. If the parameter $k\equiv0$,
then the $T_i, 1 \leq i \leq d$ are reduced to the partial
derivatives $\frac{\partial}{\partial x_i}, 1 \leq i \leq d$.
Therefore Dunkl analysis can be viewed as a generalization of
classical Fourier analysis (see next section).

Inspired by the definition of usual atomic Hardy spaces, we call
$L_k^{2}$-atom any function $a$ satisfying:\\ There exists a ball
$B$ of $\mathbb{R}^{d}$ such that
 \begin{itemize}
 \item[i)] $ Supp(a)\subset B$.
 \item[ii)] $\Big(\displaystyle\int_{\mathbb{R}^d}|a(x)|^2 d\nu_{k}(x)\Big)^{\frac{1}{2}}\leq (\nu_{k}(B))^{-\frac{1}{2}}.$
 \item[iii)] $ \displaystyle\int_{B}a(x)d\nu_{k}(x)=0,$
 \end{itemize}
  where $\nu_k$ is the weighted Lebesgue measure associated to the
Dunkl operators defined by
\begin{eqnarray*}d\nu_k(x):=w_k(x)dx\quad
\mbox{with}\;\;w_k(x) = \prod_{\xi\in R_+} |\langle
\xi,x\rangle|^{2k(\xi)}, \quad x \in \mathbb{R}^d.\end{eqnarray*}
$R_+$ being a positive root system and $\langle .,.\rangle$ the
standard Euclidean scalar product on $\mathbb{R}^d$ (see next
section).

We denote by $L^p_{k}(\mathbb{R}^d)$, $1\leq p<+\infty $ the space
$L^{p}(\mathbb{R}^d,d\nu_k(x)).$ We define the atomic Hardy space
$H^{1}_k(\mathbb{R}^{d})$ as the space of all functions \\$f\in
L^{1}_{k}(\mathbb{R}^{d})\cap L^{2}_{k}(\mathbb{R}^{d})$
 which can be written $f=\displaystyle\sum_{j=1}^{+\infty}\lambda_{j}a_{j},$ where the series converges
 in $L^{1}_{k}(\mathbb{R}^{d})\cap L^{2}_{k}(\mathbb{R}^{d})$
 with $a_{j}$ are $L_k^{2}$-atoms and $\displaystyle\sum_{j=1}^{+\infty}|\lambda_{j}|<+\infty
 .$ If $f\in  H_k^{1}(\mathbb{R}^{d})$,
we define $$\|f\|_{H_k^{1}}=\inf
\Big\{\sum_{j=1}^{+\infty}|\lambda_{j}|:\,f=\displaystyle\sum_{j=1}^{+\infty}\lambda_{j}a_{j}\Big\}.$$
Here the infimum is taken over all atomic decompositions of $f$.\\

The aim of this paper is to investigate the action of some
fundamental operators on the atomic Hardy space
$H_k^{1}(\mathbb{R}^{d})$. We use $\|\ .\;\|_{p,k}$ as a shorthand
for $\|\ .\;\|_{L^p_k( \mathbb{R}^d)}$. First, we consider the Riesz
transforms which are the operators $\mathcal{R}_{j}^k,\;j=1,...,d$
(see [3, 16]) defined on $L^{2}_{k}(\mathbb{R}^{d})$  by
\begin{eqnarray*}
  \mathcal{R}_{j}^k(f)(x) =\lambda_{k}
   \,\lim_{\epsilon\rightarrow 0}\int_{\|y\|>\epsilon}\tau_{x}(f)(-y)\frac{y_{j}}{\|y\|^{2\gamma+d+1}}d\nu_{k}(y),x\in\mathbb{R}^{d}
\end{eqnarray*} where $\tau_x$ is the Dunkl translation operator and
\begin{eqnarray*}
 \lambda_{k}=2^{\gamma+\frac{d}{2}}\frac{\Gamma(\gamma+\frac{d+1}{2})}{\sqrt{\pi}},\quad with \,\,\gamma=
 \sum_{\xi \in R_+}k(\xi).
\end{eqnarray*}We show that
$\mathcal{R}_{j}^k $ maps $H_k^{1}(\mathbb{R}^{d})$ to
$L_k^{1}(\mathbb{R}^{d})$ and for $f$ in $H_k^{1}(\mathbb{R}^{d})$,
we have
\begin{eqnarray*}\|\mathcal{R}_{j}^k(f)\|_{1,k}\leq c\,\|f\|_{H_k^{1}}.
\end{eqnarray*} It was proved in [3] that the Riesz transform $\mathcal{R}_{j}^k$
 is of a weak-type (1,1) and it can
be extended to a bounded operator from $L_{k}^{p}(\mathbb{R}^{d})$
into it self for $1<p<+\infty$:
\begin{eqnarray}\|\mathcal{R}_{j}^k(f)\|_{p,k}\leq c\,\|f\|_{p,k}\,,\;\mbox{for}\;f\in L_{k}^{p}(\mathbb{R}^{d}).\end{eqnarray}
Second, we prove that for $f$ in $H_k^{1}(\mathbb{R}^{d})$, the
Dunkl transform  $\mathcal{F}_k$ which enjoys properties similar to
those of the classical Fourier transform (see next section),
satisfies the inequality
\begin{eqnarray*}\int_{\mathbb{R}^{d}}\|y\|^{-(2\gamma+d)}|\mathcal{F}_{k}(f)(y)|d\nu_{k}(y)\leq
c\,\|f\|_{H_k^{1}}.\end{eqnarray*} The case $1<p\leq2$ was shown for
$f$ in $L_k^{p}(\mathbb{R}^{d})$ (see [1], Section 4, Lemma 1) and
gives the Hardy-Littlewood-Paley inequality
\begin{eqnarray*}\Big(\int_{\mathbb{R}^{d}}\|x\|^{(2\gamma+d)(p-2)}|\mathcal{F}_{k}(f)(x)|^{p}
d\nu_{k}(x)\Big)^{\frac{1}{p}} \leq c\,\|f\|_{p,k}.\end{eqnarray*}
Finally, we establish that for all $f$ in $H_k^{1}(\mathbb{R}^{d})$,
we have
\begin{eqnarray*}
\|\mathcal{H}_kf\|_{1,k}\leq 2\,\|f\|_{H_k^{1}},
\end{eqnarray*}
where $\mathcal{H}_k$ is a Hardy-type averaging operator given by
$$\mathcal{H}_kf(x)=\frac{1}{\nu_{k}(B_{\|x\|})}\int_{B_{\|x\|}}f(y)d\nu_{k}(y),$$
with $B_{\|x\|}=B(0,\|x\|)$ is the ball of radius $\|x\|$  centered
at $0$. The case $1<p<+\infty$ was obtained for $f$ in
$L_k^{p}(\mathbb{R}^{d})$ (see [2]) and gives the weighted Hardy
inequality
\begin{eqnarray} \Big(\int_{\mathbb{R}^d}
\Big(\frac{1}{\nu_k(B_{\|x\|})}\int_{B_{\|x\|}}|f(y)|
d\nu_k(y)\Big)^p d\nu_k(x)\Big)^\frac{1}{p} \leq \frac{p}{p-1}
\,\|f\|_{p,k}.\end{eqnarray}This is an extension to the Dunkl theory of the results  obtained in the classical harmonic analysis (see [9, 10, 14]).\\

The contents of this paper are as follows. \\In section 2, we
collect some basic definitions and results about harmonic analysis
associated with the rational Dunkl operators. \\
In section 3, we investigate in three subsection, the action
respectively of the Riesz transforms
$\mathcal{R}_{j}^k,\,j=1,...,d$, the Dunkl transform
$\mathcal{F}_{k}$ and the Hardy-type averaging
operator $\mathcal{H}_k$, on the atomic Hardy spaces  $H_k^{1}(\mathbb{R}^{d})$.\\

Along this paper, we denote $\langle .,.\rangle$ the usual Euclidean
inner product in $\mathbb{R}^d$ as well as its extension to
$\mathbb{C}^d \times\mathbb{C}^d$, we write for $x \in
\mathbb{R}^d,$ $\|x\| = \sqrt{\langle x,x\rangle}$ and we use $c$ to
represent a suitable positive constant which is not necessarily the
same in each occurrence. Furthermore, we denote by

$\bullet\quad \mathcal{E}(\mathbb{R}^d)$ the space of infinitely
differentiable functions on $\mathbb{R}^d$.

$\bullet\quad \mathcal{S}(\mathbb{R}^d)$ the Schwartz space of
functions in $\mathcal{E}( \mathbb{R}^d)$ which are rapidly
decreasing as well as their derivatives.

$\bullet\quad \mathcal{D}(\mathbb{R}^d)$ the subspace of
$\mathcal{E}(\mathbb{R}^d)$ of compactly supported functions.
\section{Preliminaries}
 $ $ In this section, we recall some notations and
results in Dunkl
theory and we refer for more details to [7, 8, 12] and the surveys [13].\\

Let $G$ be a finite reflection group on $\mathbb{R}^{d}$, associated
with a root system $R$. For $\alpha\in R$, we denote by
$\mathbb{H}_\alpha$ the hyperplane orthogonal to $\alpha$. For a
given $\beta\in\mathbb{R}^d\backslash\bigcup_{\alpha\in R}
\mathbb{H}_\alpha$, we fix a positive subsystem $R_+=\{\alpha\in R:
\langle \alpha,\beta\rangle>0\}$. We denote by $k$ a nonnegative
multiplicity function defined on $R$ with the property that $k$ is
$G$-invariant. We associate with $k$ the index
$$\gamma = \sum_{\xi \in R_+} k(\xi) \geq 0,$$
and a weighted measure $\nu_k$ given by
\begin{eqnarray*}d\nu_k(x):=w_k(x)dx\quad
\mbox{ where }\;\;w_k(x) = \prod_{\xi\in R_+} |\langle
\xi,x\rangle|^{2k(\xi)}, \quad x \in \mathbb{R}^d.\end{eqnarray*}
Note that $ (\mathbb{R}^{d}, \,d\nu_k(x))$ is a space of homogeneous
type, that is, there exists a constant $c>0$ such that
\begin{eqnarray}
  \nu_{k}(B(x,2r)) \leq c\, \nu_{k}(B(x,r)),\quad\forall\,x\in
  \mathbb{R}^{d},\:r>0,
\end{eqnarray}where $B(x,r)$ is the closed ball of radius $r$ centered at $x$ (see[14], Ch.1).

For every $1 \leq p \leq + \infty$, we denote by
$L^p_k(\mathbb{R}^d)$ the spaces $L^{p}(\mathbb{R}^d, d\nu_k(x)),$
 and $L^p_k( \mathbb{R}^d)^{rad}$ the subspace of those $f \in L^p_k(
\mathbb{R}^d)$ that are radial. We use $\|\ \;\|_{p,k}$ as a
shorthand for $\|\ \;\|_{L^p_k( \mathbb{R}^d)}$.

We introduce the Mehta-type constant $c_k$ by
$$c_k = \left(\int_{\mathbb{R}^d} e^{- \frac{\|x\|^2}{2}}
w_k (x)dx\right)^{-1}.$$ By using the homogeneity of degree
$2\gamma$ of $w_k$, it was shown in [8] that for a function $f$ in
$L^1_k ( \mathbb{R}^d)^{rad}$, there exists a function $F$ on $[0, +
\infty)$ such that $f(x) = F(\|x\|)$, for all $x \in \mathbb{R}^d$.
The function $F$ is integrable with respect to the measure
$r^{2\gamma+d-1}dr$ on $[0, + \infty)$ and we have
 \begin{eqnarray} \int_{\mathbb{R}^d}  f(x)\,d\nu_k(x)&=&\int^{+\infty}_0
\Big( \int_{S^{d-1}}f(ry)w_k(ry)d\sigma(y)\Big)r^{d-1}dr\nonumber\\
&=&
 \int^{+\infty}_0
\Big( \int_{S^{d-1}}w_k(ry)d\sigma(y)\Big)
F(r)r^{d-1}dr\nonumber\\&= & d_k\int^{+ \infty}_0 F(r)
r^{2\gamma+d-1}dr,
\end{eqnarray}
  where $S^{d-1}$
is the unit sphere on $\mathbb{R}^d$ with the normalized surface
measure $d\sigma$  and \begin{eqnarray*}d_k=\int_{S^{d-1}}w_k
(x)d\sigma(x) = \frac{c^{-1}_k}{2^{\gamma +\frac{d}{2} -1}
\Gamma(\gamma + \frac{d}{2})}\,.  \end{eqnarray*} In particular,
\begin{eqnarray*}
\nu_{k}(B(0,R))=\frac{d_{k}}{2\gamma+d}R^{2\gamma+d}.
\end{eqnarray*}

The Dunkl operators $T_j\,,\ \ 1\leq j\leq d\,$, on $\mathbb{R}^d$
associated with the reflection group $G$ and the multiplicity
function $k$ are the first-order differential-difference operators
given by
$$T_jf(x)=\frac{\partial f}{\partial x_j}(x)+\sum_{\alpha\in R_+}k(\alpha)
\alpha_j\,\frac{f(x)-f(\rho_\alpha(x))}{\langle\alpha,x\rangle}\,,\quad
f\in\mathcal{E}(\mathbb{R}^d)\,,\quad x\in\mathbb{R}^d\,,$$ where
$\rho_\alpha$ is the reflection on the hyperplane
$\mathbb{H}_\alpha$ and $\alpha_j=\langle\alpha,e_j\rangle,$
$(e_1,\ldots,e_d)$ being the canonical basis of $\mathbb{R}^d$.
\begin{remark}In the case $k\equiv0$, the weighted function $w_k\equiv1$ and the measure $\nu_k$ associated to the
Dunkl operators coincide with the Lebesgue measure. The $T_j$ are
reduced to the corresponding partial derivatives. Therefore, the
Dunkl theory can be viewed as a generalization of the classical
Fourier analysis.\end{remark}

For $y \in \mathbb{C}^d$, the system
$$\left\{\begin{array}{lll}T_ju(x,y)&=&y_j\,u(x,y),\qquad1\leq j\leq d\,,\\  &&\\
u(0,y)&=&1\,.\end{array}\right.$$ admits a unique analytic solution
on $\mathbb{R}^d$, denoted by $E_k(x,y)$ and called the Dunkl
kernel. This kernel has a unique holomorphic extension to
$\mathbb{C}^d \times \mathbb{C}^d $.\\
 We have for all $\lambda\in
\mathbb{C}$ and $z, z'\in \mathbb{C}^d,\;
 E_k(z,z') = E_k(z',z)$,  $E_k(\lambda z,z') = E_k(z,\lambda z')$.\\
(See [13]) For all $x\in \mathbb{R}^{d}$, $y\in \mathbb{C}^{d}$ and
all multi-indices $\alpha =(\alpha_{1}, \alpha_{2},...,
\alpha_{d})\in \mathbb{Z}_{+}^{d},$
\begin{eqnarray*}|\partial_{y}^{\alpha}E_{k}(x,y)|\leq \|x\|^{|\alpha|}\max_{g\in G}e^{\mathcal{R}e<g.x,y>},
\end{eqnarray*}where $|\alpha|=\displaystyle{\sum_{j=1}^{d}\alpha_{j}}.$
In particular, for all $x,y\in \mathbb{R}^{d},$
\begin{eqnarray}
|E_{k}(-ix,y)|\leq 1 \qquad \mbox{and}\qquad
|\frac{\partial}{\partial y_{j}}E_{k}(-ix,y)|\leq \|x\|, \quad
\forall\,j=1,...,d.
\end{eqnarray}
The Dunkl transform $\mathcal{F}_k$ is defined for $f \in
\mathcal{D}( \mathbb{R}^d)$ by
$$\mathcal{F}_k(f)(x) =c_k\int_{\mathbb{R}^d}f(y) E_k(-ix, y)d\nu_k(y),\quad
x \in \mathbb{R}^d.$$  We list some known properties of this
transform:
\begin{itemize}
\item[i)] If $f$ is in $L^1_k( \mathbb{R}^d)$, then $\mathcal{F}_k(f)$ is in $\mathcal{C}_0(\mathbb{R}^{d})$
where $\mathcal{C}_0(\mathbb{R}^{d})$ denotes the space of
continuous functions on $ \mathbb{R}^{d}$ which vanish at infinity
and we have
\begin{eqnarray*}\| \mathcal{F}_k(f)\|_{\infty} \leq
 \|f\|_{ 1,k}\;. \end{eqnarray*}
\item[ii)] The Dunkl transform is an automorphism on the Schwartz space $\mathcal{S}(\mathbb{R}^d)$.
\item[iii)] When both $f$ and $\mathcal{F}_k(f)$ are in $L^1_k( \mathbb{R}^d)$,
 we have the inversion formula \begin{eqnarray*} f(x) =   \int_{\mathbb{R}^d}\mathcal{F}_k(f)(y) E_k( ix, y)d\nu_k(y),\quad
x \in \mathbb{R}^d.\end{eqnarray*}
\item[iv)] (Plancherel's theorem) The Dunkl transform on $\mathcal{S}(\mathbb{R}^d)$
 extends uniquely to an isometric automorphism on
$L^2_k(\mathbb{R}^d)$.
\end{itemize}
K. Trim\`eche has introduced in [17] the Dunkl translation operators
$\tau_x$, $x\in\mathbb{R}^d$, on $\mathcal{E}( \mathbb{R}^d)\;.$ For
$f\in \mathcal{S}( \mathbb{R}^d)$ and $x, y\in\mathbb{R}^d$, we have
\begin{eqnarray*}\mathcal{F}_k(\tau_x(f))(y)=E_k(i x, y)\mathcal{F}_k(f)(y).\end{eqnarray*}  Notice that for all $x,y\in\mathbb{R}^d$,
$\tau_x(f)(y)=\tau_y(f)(x)$ and for fixed $x\in\mathbb{R}^d,$
\begin{eqnarray*}\tau_x \mbox{\; is a continuous linear mapping from \;}
\mathcal{E}( \mathbb{R}^d) \mbox{\;
into\;}\mathcal{E}(\mathbb{R}^d)\,.\end{eqnarray*} As an operator on
$L_k^2(\mathbb{R}^d)$, $\tau_x$ is bounded. According to ([15],
Theorem 3.7), the operator $\tau_x$ can be extended to the space of
radial functions $L^p_k(\mathbb{R}^d)^{rad},$ $1 \leq p \leq 2$ and
we have for a function $f$ in $L^p_k(\mathbb{R}^d)^{^{rad}}$,
\begin{eqnarray*}\|\tau_x(f)\|_{p,k} \leq \|f\|_{p,k}.\end{eqnarray*}

(see [16]) The Riesz transforms $\mathcal{R}_{j}^k$, $j=1...d,$ are
multiplier operators given by
\begin{eqnarray}
\mathcal{F}_{k}(\mathcal{R}_{j}^k(f))(\xi) =
\frac{-i\xi_{j}}{\|\xi\|}\,\mathcal{F}_{k}(f)(\xi),\quad\mbox{for}\;f\in
  \mathcal{S}(\mathbb{R}^{d}).
\end{eqnarray}
It was proved in [3] that for  $f\in L^{2}_k(\mathbb{R}^{d})$ with
compact support and for all $x\in \mathbb{R}^{d}$
  such that $g.x\notin\,supp(f),$ $\forall\;g\in G,$ we have
 \begin{eqnarray}
 \mathcal{R}_{j}^k(f)(x)=\int_{\mathbb{R}^{d}}\mathcal{K}_{j}(x,y)f(y)d\nu_{k}(y),\end{eqnarray}
  where the kernel $\mathcal{K}_{j},$ $j=1...d,$ satisfy
  \begin{eqnarray}
    \int_{\displaystyle{\min_{g\in G}}\|g.x-y\|>2\|y-y_{0}\|}|\mathcal{K}_{j}(x,y)-\mathcal{K}_{j}(x,y_{0})|d\nu_{k}(x)
     \leq c\,,\quad\mbox{for}\; y,y_{0} \in \mathbb{R}^{d}.
  \end{eqnarray} Using the Calder\'{o}n-Zygmund
decomposition, the authors in [3] showed that the Riesz transform
$\mathcal{R}_{j}^k$ is of a weak-type (1,1) and it can be extended
to a bounded operator from $L_{k}^{p}(\mathbb{R}^{d})$ into it self
for $1<p<+\infty$.
\section{Some results on the Hardy space $H^1_k$}
We begin with a remark and an example.
\begin{remark}
By the definition given in the introduction, if $a$ is an
$L_k^{2}$-atom, then
$$\|a\|_{1,k}=\int_{B}|a(y)|d\nu_{k}(y)\leq
(\nu_{k}(B))^{\frac{1}{2}}\|a\|_{2,k}\leq 1.$$
\end{remark}
\begin{example} Fix $h\in L^{2}_{k}(\mathbb{R}^{d})$ with $\|h\|_{2,k}=
1$ and let $(B_{j})_{j \in \mathbb{N}}$ be a family of balls such
that
$$B_{j}\subset C_{j}=\Big\{x\in \mathbb{R}^{d};
4^{-\frac{j+1}{2\gamma+d}}\leq \|x\|\leq
4^{-\frac{j}{2\gamma+d}}\Big\}.$$
 The average of $h$ over $B_{j}$ is given by $$Avgh=\displaystyle\frac{1}{\nu_{k}(B_{j})}\int_{B_{j}}h(x)d\nu_{k}(x).$$
Put
$a_{j}=(\nu_{k}(B_{j}))^{-\frac{1}{2}}\displaystyle\Big(h-Avgh\Big)\chi_{B_{j}},$
then each $a_{j}$ is an $L_k^{2}$-atom and the function
$f=\displaystyle\sum_{j=1}^{+\infty}(\nu_{k}(B_{j}))^{\frac{1}{2}}a_{j}$
is in $H_k^{1}(\mathbb{R}^{d})$ with
\begin{eqnarray*}\|f\|_{H_k^{1}}\leq\sum_{j=
1}^{+\infty}(\nu_{k}(B_{j}))^{\frac{1}{2}}\leq \frac{
\sqrt{d_{k}}}{\sqrt{2\gamma+d}}.\end{eqnarray*} Indeed, $a_{j}$
satisfies $supp(a_{j})\subset B_{j},$ $\displaystyle\int_{B_{j}}
a_{j}(x)d\nu_{k}(x)=0$ and\begin{eqnarray*} \|a_{j}\|_{2,k}^{2}
&= & (\nu_{k}(B_{j}))^{-1}\int_{B_{j}} \Big|h(x)-Avgh\Big|^{2}d\nu_{k}(x) \\
 &= & (\nu_{k}(B_{j}))^{-1}\int_{B_{j}} (h(x))^{2}d\nu_{k}(x)+\nu_{k}(B_{j}) (Avgh)^{2}-2Avgh\int_{B_{j}} h(x)d\nu_{k}(x)\\
 &= & (\nu_{k}(B_{j}))^{-1}\int_{B_{j}} (h(x))^{2}d\nu_{k}(x)-\nu_{k}(B_{j}) (Avgh)^{2}\\
 &\leq & (\nu_{k}(B_{j}))^{-1}\int_{B_{j}} (h(x))^{2}d\nu_{k}(x)\leq
 (\nu_{k}(B_{j}))^{-1}.
  \end{eqnarray*}
  Then we obtain,
  \begin{eqnarray*}\|f\|_{H_k^{1}}\leq \sum_{j=1}^{+\infty}(\nu_{k}(B_{j}))^{\frac{1}{2}}&\leq&\sum_{j= 1}^{+\infty}\Big[\nu_{k}
  (B(0,4^{-\frac{j}{2\gamma+d}}))\Big]^{\frac{1}{2}}\\&=& \frac{\sqrt{d_{k}}}{\sqrt{2\gamma+d}}\sum_{j=1}^{+\infty}2^{-j}=\frac{\sqrt{d_{k}}}
  {\sqrt{2\gamma+d}}.\end{eqnarray*}
\end{example}
\subsection{Riesz transform on the Hardy space}
In this section, we prove that the Riesz transforms are bounded
operators from $H^1_k(\mathbb{R}^d)$ to $L^1_k(\mathbb{R}^d)$.
Before, we need to prove a useful lemma.
\begin{lemma}
There exists a constant $c> 0$ such that for all $L^{2}_k$-atom $a$,
we have\begin{eqnarray*}\|\mathcal{R}_{j}(a)\|_{1,k}\leq c, \quad
1\leq j \leq d.\end{eqnarray*}
\end{lemma}
\begin{proof}
Let $a$ be an $L^{2}_k$-atom supported in
$B=B(y_{0},r),\,y_{0}\in\mathbb{R}^{d}.$ If we put
$B^{\ast}=B(y_{0},4r)$ and $ Q=\displaystyle
\displaystyle\bigcup_{g\in G}g.B^{\ast},$ then we can write
\begin{eqnarray*}\int_{\mathbb{R}^{d}}|\mathcal{R}_{j}^k(a)(x)|d\nu_{k}(x)=
\int_{Q}|\mathcal{R}_{j}^k(a)(x)|d\nu_{k}(x)+\int_{Q^{c}}|\mathcal{R}_{j}^k(a)(x)|d\nu_{k}(x).\end{eqnarray*}
From (1.1), the Riesz transform $\mathcal{R}_{j}^k$ is of
strong-type $(2,2)$, then using H\"{o}lder's inequality, we obtain
\begin{eqnarray*}
  \int_{Q}|\mathcal{R}_{j}^k(a)(x)|d\nu_{k}(x) &\leq&
  (\nu_{k}(Q))^{\frac{1}{2}}\Big(\int_{Q}|\mathcal{R}_{j}^k(a)(x)|^{2}d\nu_{k}(x)\Big)^{\frac{1}{2}} \\
   &\leq& c \,(\nu_{k}(Q))^{\frac{1}{2}}\Big(\int_{B}|a(x)|^{2}d\nu_{k}(x)\Big)^{\frac{1}{2}}\\
   &\leq& c\,(\nu_{k}(Q))^{\frac{1}{2}}(\nu_{k}(B))^{-\frac{1}{2}}.
\end{eqnarray*}
From (2.1) we have, $\displaystyle \nu_{k}(Q)\leq \sum_{g\in
G}\nu_{k}(g.B^{*})\leq c\,|G|\,\nu_{k}(B),$ where $|G|$ is the order
of $G$. This gives that
\begin{eqnarray} \int_{Q}|\mathcal{R}_{j}^k(a)(x)|d\nu_{k}(x)\leq
c.\end{eqnarray} Now, if $x\notin Q$ then $\displaystyle \min_{g\in
G}\|g.x-y\|>2\|y-y_{0}\|,\,\,\forall y\in B.$  Using the integral
representation (2.5) for the Riesz transform $\mathcal{R}_{j}^k$ and
with the condition iii) for an $L^{2}_k$-atom, we have
\begin{eqnarray*}
  \int_{Q^{c}}|\mathcal{R}_{j}^k(a)(x)|d\nu_{k}(x) &=& \int_{Q^{c}}\Big|\int_{B}\mathcal{K}_{j}(x,y)a(y)d\nu_{k}(y)\Big|d\nu_{k}(x) \\
   &=& \int_{Q^{c}}\Big|\int_{B}\Big(\mathcal{K}_{j}(x,y)-\mathcal{K}_{j}(x,y_{0})\Big)a(y)d\nu_{k}(y)\Big|d\nu_{k}(x) \\
   &\leq&
   \int_{B}\int_{Q^{c}}\Big|\mathcal{K}_{j}(x,y)-\mathcal{K}_{j}(x,y_{0})\Big|d\nu_{k}(x)|a(y)|d\nu_{k}(y).
   \end{eqnarray*}
   Since $y\in B$, we obtain from (2.6)\\
   $ \displaystyle\int_{Q^{c}}\Big|\mathcal{K}_{j}(x,y)-\mathcal{K}_{j}(x,y_{0})\Big|d\nu_{k}(x)$
    \begin{eqnarray*}
    \leq\int_{\displaystyle \min_{g\in G}\|g.x-y\|>2\|y-y_{0}\|}
    \Big|\mathcal{K}_{j}(x,y)-\mathcal{K}_{j}(x,y_{0})\Big|d\nu_{k}(x)\leq c,\end{eqnarray*}
    and by the remark 3.1, we deduce that
   \begin{eqnarray}
   \int_{Q^{c}}|\mathcal{R}_{j}^k(a)(x)|d\nu_{k}(x)\leq c\int_{B}|a(y)|d\nu_{k}(y)\leq c.
\end{eqnarray}
Combining (3.1) and (3.2), we conclude that
\begin{eqnarray*}\|\mathcal{R}_{j}^k(a)\|_{1,k}\leq
c.\end{eqnarray*}
\end{proof}
\begin{theorem}
 The Riesz transforms $\mathcal{R}_{j}^k$, $1\leq j\leq d$ map $H^{1}_k(\mathbb{R}^{d})$ to
 $L^{1}_{k}(\mathbb{R}^{d})$ and we have for all $f\in
H^{1}_k(\mathbb{R}^{d})$,
\begin{eqnarray*}\|\mathcal{R}_{j}^k(f)\|_{1,k}\leq c\,\|f\|_{H^{1}_k}.\end{eqnarray*}
  \end{theorem}
\begin{proof} Let $\displaystyle f \in
H^{1}_k(\mathbb{R}^d),$ $\displaystyle
f=\sum_{i=1}^{+\infty}\lambda_{i}a_{i}.$ For $\rho >0$ and $1\leq j
\leq d$, we have
\begin{eqnarray} \nu_{k}\Big(\Big\{\mathcal{R}_{j}^k(f)-\sum_{i=1}^{+\infty}\lambda_{i}
\mathcal{R}_{j}^k(a_{i})>\rho \Big\}\Big) &\leq&
\nu_{k}\Big(\Big\{\Big|\mathcal{R}_{j}^k(f)-\sum_{i=1}^{N}
\lambda_{i}\mathcal{R}_{j}^k(a_{i})\Big|>\frac{\rho}{2}\Big\}\Big)\nonumber\\&+&\nu_{k}
\Big(\Big\{\Big|\sum_{i=N+1}^{+\infty}\lambda_{i}\mathcal{R}_{j}^k(a_{i})\Big|>\frac{\rho}{2}\Big\}\Big).
\end{eqnarray}
Since $\mathcal{R}_{j}^k$ is of weak-type $(1,1)$, we assert that
\begin{eqnarray*} \nu_{k}\Big(\Big\{\Big|\mathcal{R}_{j}^k(f)-\sum_{i=1}^{N}
\lambda_{i}\mathcal{R}_{j}^k(a_{i})\Big|>\frac{\rho}{2}\Big\}\Big)\leq
\frac{2}{\rho}\,c\,\Big\|f-\sum_{i=1}^{N}\lambda_{i}a_{i}\Big\|_{1,k}\,,\end{eqnarray*}
and clearly \begin{eqnarray*} \nu_{k}
\Big(\Big\{\Big|\sum_{i=N+1}^{+\infty}\lambda_{i}\mathcal{R}_{j}^k(a_{i})\Big|>\frac{\rho}{2}\Big\}\Big)\leq
\frac{2}{\rho}\Big\|\sum_{i=N+1}^{+\infty}\lambda_{i}\mathcal{R}_{j}^k(a_{i})\Big\|_{1,k}\leq
\frac{2}{\rho}\,c\sum_{i=N+1}^{+\infty}|\lambda_{i}|.\end{eqnarray*}
Now using the fact that
$\displaystyle\sum_{i=1}^{N}\lambda_{i}a_{i}$ converges to $f$ in
$L^{1}_k(\mathbb{R}^d)$
 and $\displaystyle\sum_{i=1}^{+\infty}|\lambda_{i}|<+\infty$, we have both terms in the sum in (3.3)
  converge to zero as
 $N\rightarrow +\infty$. Hence, we deduce that \begin{eqnarray*}
\nu_{k}\Big(\Big\{\mathcal{R}_{j}^k(f)-\sum_{i=1}^{+\infty}\lambda_{i}\mathcal{R}_{j}^k(a_{i})>\rho
\Big\}\Big)=0,\;\mbox{for all}\; \rho >0,\end{eqnarray*} which gives
that
$\mathcal{R}_{j}^k(f)=\displaystyle\sum_{i=1}^{+\infty}\lambda_{i}\mathcal{R}_{j}^k(a_{i})\quad
a.e.$ Using the lemma 3.1, we can assert that
\begin{eqnarray*}\|\mathcal{R}_{j}^k(f)\|_{1,k}\leq \sum_{i=1}^{+\infty}|\lambda_{i}|\, \|\mathcal{R}_{j}^k(a_{i})\|_{1,k}
\leq c\,\sum_{i=1}^{+\infty}|\lambda_{i}|,\end{eqnarray*}which gives
$\|\mathcal{R}_{j}^k(f)\|_{1,k}\leq c\,\|f\|_{H^{1}_k}.$ This proves
our result.
\end{proof}
\begin{cor}
If $f\in H^{1}_k(\mathbb{R}^{d})$ then $\displaystyle
\int_{\mathbb{R}^{d}}f(x)d\nu_{k}(x)=0.$
\end{cor}
\begin{proof}
By Theorem 3.1, if $f\in H^{1}_k(\mathbb{R}^{d})$ then
$\mathcal{R}_{j}^k(f)\in L^{1}_{k}(\mathbb{R}^{d})$, $ 1\leq j \leq
d,$ then $\mathcal{F}_{k}(\mathcal{R}_{j}^k(f))$ is in
$\mathcal{C}_0(\mathbb{R}^{d})$. From (2.4), it follows that
$\mathcal{F}_{k}(\mathcal{R}_{j}^k(f))$ is continuous at zero if and
only if $\mathcal{F}_{k}(f)(0)=0$, which gives that $\displaystyle
\int_{\mathbb{R}^{d}}f(x)d\nu_{k}(x)=0.$
\end{proof}
\subsection{Dunkl transform on the Hardy space}
In this section, we study the Dunkl transform on the space
$H^{1}_k(\mathbb{R}^{d}).$ Before, we need to prove a useful result.
\begin{lemma}
There exists a positive constant $c>0$ such that for all
$L^{2}_k$-atom $a$ supported in the ball $B=B(0,R),$  and all $y \in
\mathbb{R}^{d}$, we have
\begin{eqnarray*}
|\mathcal{F}_{k}(a)(y)|\leq c\,R\,\|y\|.
\end{eqnarray*}
\end{lemma}
\begin{proof}
By virtue of the condition iii) for an $L^{2}_k$-atom, we have for
$y \in \mathbb{R}^{d}$
\begin{eqnarray*}
\mathcal{F}_{k}(a)(y)=c_{k}\int_{B}[E_{k}(-ix,y)-1]\,a(x)\,d\nu_{k}(x).\end{eqnarray*}
Using mean value theorem and (2.3), we get for $x\in B$
\begin{eqnarray*}
|E_{k}(-ix,y)-1|&=&\Big|\sum_{j=1}^{d}y_{j}\int_{0}^{1}\frac{\partial E_{k}}{\partial y_{j}}(-ix,ty)\,dt\Big|\\
&\leq & \sum_{j=1}^{d}|y_{j}|\,\|x\|\\
&\leq &\sqrt{d}\,\|y\|\,\|x\| \leq \sqrt{d}\,R\,\|y\|.
\end{eqnarray*}
According to the remark 3.1, we can assert that
\begin{eqnarray*}
|\mathcal{F}_{k}(a)(y)|  \leq
c_{k}\sqrt{d}\,R\,\|y\|\,\|a\|_{1,k}\leq c \,R\,\|y\|\,,
\end{eqnarray*} which proves the result.
\end{proof}
\begin{theorem} Let $f\in H^{1}_k(\mathbb{R}^{d})$, then we have
 \begin{eqnarray*}\int_{\mathbb{R}^{d}}\|y\|^{-(2\gamma+d)}|\mathcal{F}_{k}(f)(y)|d\nu_{k}(y)\leq c\,\|f\|_{H^{1}_k}.
 \end{eqnarray*}
\end{theorem}
\begin{proof}
Let $a$ be an $L^{2}_k$-atom supported in the ball $B=B(0,R).$ We
write
\begin{eqnarray*}\int_{\mathbb{R}^{d}}|\mathcal{F}_{k}(a)(y)|\frac{d\nu_{k}(y)}{\|y\|^{2\gamma+d}}&=&
\int_{\|y\|\leq \frac{1}{R}}|\mathcal{F}_{k}(a)(y)|\frac{d\nu_{k}(y)}{\|y\|^{2\gamma+d}}+
\int_{\|y\|>\frac{1}{R}}|\mathcal{F}_{k}(a)(y)|\frac{d\nu_{k}(y)}{\|y\|^{2\gamma+d}}\\
&=&I_{1}+I_{2}.
\end{eqnarray*}
Using the lemma 3.2 and (2.2), we can estimate $I_{1}$ as follows
\begin{eqnarray}
I_{1}&\leq & c \,R\int_{\|y\|\leq \frac{1}{R}}\|y\|^{-(2\gamma+d-1)}d\nu_{k}(y)\nonumber\\
&=&c \,R\, d_{k}\int_{0}^{\frac{1}{R}}r^{-(2\gamma+d-1)}\, r^{2\gamma+d-1}dr\nonumber\\
&=&c\, d_{k}\,.
\end{eqnarray}
To estimate $I_{2}$, we use H\"{o}lder's inequality, Plancherel's
theorem for the Dunkl transform $\mathcal{F}_{k},$ and (2.2)
\begin{eqnarray}
I_{2}&\leq & \Big(\int_{\|y\|>\frac{1}{R}}|\mathcal{F}_{k}(a)(y)|^{2}d\nu_{k}(y)\Big)^{\frac{1}{2}}
\Big(\int_{\|y\|>\frac{1}{R}}\|y\|^{-2(2\gamma+d)}d\nu_{k}(y)\Big)^{\frac{1}{2}}\nonumber\\
&\leq & \|a\|_{2,k}\sqrt{d_{k}}\Big(\int_{\frac{1}{R}}^{+\infty}r^{-2(2\gamma+d)}\,r^{2\gamma+d-1}\Big)^{\frac{1}{2}}dr\nonumber
\\
&\leq & (\nu_{k}(B))^{-\frac{1}{2}}\frac{\sqrt{d_{k}}}{\sqrt{2\gamma+d}}\,\Big(\frac{1}{R}\Big)^{-\frac{2\gamma+d}{2}}\nonumber
\\
&\leq &
\frac{\sqrt{2\gamma+d}}{\sqrt{d_{k}}}\,R^{-\frac{2\gamma+d}{2}}\frac{\sqrt{d_{k}}}{\sqrt{2\gamma+d}}\,
\Big(\frac{1}{R}\Big)^{-\frac{2\gamma+d}{2}}\nonumber\\
&=&1.
\end{eqnarray}
Hence, from (3.3) and (3.4), we obtain
\begin{eqnarray}\int_{\mathbb{R}^{d}}|\mathcal{F}_{k}(a(y))|\frac{d\nu_{k}(y)}{\|y\|^{2\gamma+d}}\leq c.\end{eqnarray}
Since $\mathcal{F}_{k}$ is a continuous linear mapping from
$L^{1}_{k}(\mathbb{R}^{d})$ to $\mathcal{C}_0(\mathbb{R}^{d})$, we
deduce that for $f\in H^{1}_k(\mathbb{R}^{d})$, $\displaystyle
f=\sum_{j=1}^{+\infty}\lambda_{j}a_{j}$, we have
$\mathcal{F}_{k}(f)(y)=\displaystyle
\sum_{j=1}^{+\infty}\lambda_{j}\mathcal{F}_{k}(a_{j})(y)\,\;
a.e.\,\;y\in \mathbb{R}^{d}.$ Using (3.6), this gives
\begin{eqnarray*}
\int_{\mathbb{R}^{d}}\|y\|^{-(2\gamma+d)}|\mathcal{F}_{k}(f(y))| d\nu_{k}(y) &\leq &\sum_{j=1}^{+\infty}|\lambda_{j}| \int_{\mathbb{R}^{d}}|\mathcal{F}_{k}(a_j(y))|\frac{d\nu_{k}(y)}{\|y\|^{2\gamma+d}}\\
&\leq &c\sum_{j=1}^{+\infty}|\lambda_{j}|,
\end{eqnarray*}
thus the proof is completed.
\end{proof}
\subsection{Hardy inequality for $H^{1}_k(\mathbb{R}^{d})$} In this section, we give a Hardy-type inequality on
$H^{1}_k(\mathbb{R}^{d}).$ Recall that the Hardy-type averaging
operator $\mathcal{H}_k$ is given by
\begin{eqnarray*}\mathcal{H}_kf(x)=\frac{1}{\nu_{k}(B_{\|x\|})}\int_{B_{\|x\|}}f(y)d\nu_{k}(y)\quad
with \quad B_{\|x\|}=B(0,\|x\|).\end{eqnarray*}In order to establish
the result, we need to show the following lemma.
\begin{lemma}
For every $L^{2}_k$-atom $a$ ,we have$$\|\mathcal{H}_k
a\|_{1,k}\leq2.$$
\end{lemma}
\begin{proof}
Let $a$ be an $L^{2}_k$-atom on $\mathbb{R}^{d}$, supported in the
ball $B=B(0,R),$ then $supp\,(\mathcal{H}_ka)\subset B$. Indeed, if
$\|x\|>R$ then $B_{\|x\|}\cap B=B$ and $\mathcal{H}_ka=0.$ We have,
\begin{eqnarray*}
      \|\mathcal{H}_ka\|_{1,k} &\leq& \int_{B}\Big(\frac{1}{\nu_{k}(B_{\|x\|})}\int_{B_{\|x\|}}|a(y)|d\nu_{k}(y)\Big)d\nu_{k}(x) \\
       &\leq& \int_{B}\frac{1}{\nu_{k}(B_{\|x\|})}\Big(\int_{B_{\|x\|}}|a(y)|^{2}d\nu_{k}(y)\Big)^{\frac{1}{2}}\Big(\nu_{k}(B_{\|x\|})\Big)^{\frac{1}{2}}d\nu_{k}(x) \\
      &\leq& \|a\|_{2,k}\int_{B}\Big(\nu_{k}(B_{\|x\|})\Big)^{-\frac{1}{2}}d\nu_{k}(x).
       \end{eqnarray*}
By the property ii) of $a$ and (2.2), we deduce that
\begin{eqnarray*}
  \|\mathcal{H}_ka\|_{1,k}  &\leq& \Big(\frac{d_{k}}{2\gamma+d}R^{2\gamma+d}\Big)^{-\frac{1}{2}}\int_{B}\Big(\frac{d_{k}}{2\gamma+d}\|x\|^{2\gamma+d}\Big)^{-\frac{1}{2}}d\nu_{k}(x) \\
   &\leq& \Big(\frac{d_{k}}{2\gamma+d}\Big)^{-1} R^{-\frac{2\gamma+d}{2}}d_{k}\int_{0}^{R}r^{-\frac{2\gamma+d}{2}}.r^{2\gamma+d-1}dr\\
   &=&(2\gamma+d)R^{-\frac{2\gamma+d}{2}}\int_{0}^{R}r^{\frac{2\gamma+d}{2}-1}dr\,=\,2.
\end{eqnarray*}
\end{proof}
\begin{remark}
Let $FH^{1}_k(\mathbb{R}^{d})$ be the set of all finite linear
combinations of $L^{2}_k$-atoms. Using the linearity of
$\mathcal{H}_k$ and the lemma 3.3, it's easy to see that if $f\in
FH^{1}_k(\mathbb{R}^{d}),$
$f=\displaystyle\sum_{j=1}^{N}\lambda_{j}a_{j}$,
 then
 $\displaystyle\mathcal{H}_kf=\sum_{j=1}^{N}\lambda_{j}\mathcal{H}_ka_{j}$
 and $\displaystyle\|\mathcal{H}_kf\|_{1,k}\leq
2\sum_{j=1}^{N}|\lambda_{j}|,$ which implies that
$\|\mathcal{H}_kf\|_{1,k}\leq  2\,\|f\|_{H^{1}_k}.$
\end{remark}
\begin{theorem}
For all $f \in  H^{1}_k(\mathbb{R}^{d})$, we have
\begin{eqnarray*}
\|\mathcal{H}_kf\|_{1,k}\leq 2\,\|f\|_{H^{1}_k}.
\end{eqnarray*}
\end{theorem}
\begin{proof} For $f\in H^{1}_k(\mathbb{R}^{d})$, $\displaystyle
f=\sum_{j=1}^{+\infty}\lambda_{j}a_{j}$ and a positive integer $N$,
we have from the linearity of $\mathcal{H}_k$ and (1.2)
\begin{eqnarray*}\Big\|\mathcal{H}_kf-\sum_{j=1}^{N}\lambda_{j}\mathcal{H}_ka_{j}\Big\|_{2,k}=\Big\|\mathcal{H}_k(f-\sum_{j=1}^{N}\lambda_{j}a_{j})\Big\|_{2,k}
\leq 2\,\Big\|f-\sum_{j=1}^{N}\lambda_{j}a_{j}\Big\|_{2,k},
 \end{eqnarray*}which converges to zero as $N\rightarrow +\infty.$ This implies
  that $\displaystyle\mathcal{H}_kf=\sum_{j=1}^{+\infty}\lambda_{j}\mathcal{H}_ka_{j},
 \;\, a.e.$  \\
Finally, using Lemma 3.3, we can assert that
\begin{eqnarray*}
 \|\mathcal{H}_kf\|_{1,k}\leq\sum_{j=1}^{+\infty}|\lambda_{j}|\,\|\mathcal{H}_ka_{j}\|_{1,k}\leq
2\sum_{j=1}^{+\infty}|\lambda_{j}|,
\end{eqnarray*}which gives $\|\mathcal{H}_kf\|_{1,k}\leq 2\,\|f\|_{H^{1}_k}.$  This completes the proof.
\end{proof}

% ----------------------------------------------------------------
\end{document}